\documentclass{gtart}

\usepackage[inner=20mm, outer=20mm, textheight=245mm]{geometry}

\usepackage{psfrag, graphicx, subfigure, epsfig,color}

\usepackage{amsmath, amssymb, latexsym, euscript, amsthm}

\usepackage{mathptmx, wrapfig}

\usepackage{mathrsfs}

\usepackage{booktabs}

\usepackage[small,bf]{caption}
\usepackage[T1]{fontenc}

\def\tri{\mathcal{T}}
\newcommand*{\dittostraight}{---\textquotedbl---}

\theoremstyle{plain}
\newtheorem{theorem}{Theorem}
\newtheorem*{theorem*}{Theorem}
\newtheorem{lemma}[theorem]{Lemma}
\newtheorem{proposition}[theorem]{Proposition}

\newtheorem*{claim*}{Claim}

\theoremstyle{definition}
\newtheorem{definition}[theorem]{Definition}
\newtheorem*{definition*}{Definition}

\theoremstyle{remark}

\numberwithin{equation}{section}

\begin{document}
 
\title{Determining the trisection genus of orientable and non-orientable PL $4$-manifolds through triangulations\footnote{An extended abstract of this paper appeared in the Proceedings of the 34th International Symposium on Computational Geometry (SoCG 2018), Budapest, June 11--14, 2018 \cite{socg}.}}
\author{Jonathan Spreer and Stephan Tillmann}

\begin{abstract}
Gay and Kirby recently introduced the concept of a trisection for arbitrary smooth, oriented closed $4$-manifolds, and with it a new topological invariant, called the trisection genus.
This paper improves and implements an algorithm due to Bell, Hass, Rubinstein and Tillmann to compute trisections using triangulations, and extends it to non-orientable 4--manifolds. 
Lower bounds on trisection genus are given in terms of Betti numbers and used to determine the trisection genus of all standard simply connected PL $4$-manifolds.
In addition, we construct trisections of small genus directly from the simplicial structure of triangulations using the Budney-Burton census of closed triangulated $4$-manifolds. These experiments include the construction of minimal genus trisections of the non-orientable $4$-manifolds $S^3 \tilde{\times} S^1$ and $\mathbb{R}P^4$.
\end{abstract}



\maketitle


\section{Introduction}

Gay and Kirby's construction of a \emph{trisection} for arbitrary smooth, oriented closed $4$-manifolds~\cite{GK} defines a decomposition of the $4$-manifold into three $4$-dimensional handlebodies\footnote{A $d$-dimensional handlebody (or, more precisely, $1$-handlebody) is the regular neighbourhood of a graph embedded into Euclidean $d$-space.} glued along their boundaries in the following way: Each handlebody is a boundary connected sum of copies of $S^1 \times B^3,$ and has boundary a connected sum of copies of $S^1 \times S^2$ (here, $B^i$ denotes the $i$-dimensional ball and $S^j$ denotes the $j$-dimensional sphere). The triple intersection of the $4$-dimensional handlebodies is a closed orientable surface $\Sigma$, called the {\em central surface}, which divides each of their boundaries into two $3$-dimensional handlebodies (and hence is a Heegaard surface). These $3$-dimensional handlebodies are precisely the intersections of pairs of the $4$-dimensional handlebodies.

A trisection naturally gives rise to a quadruple of non-negative integers $(g; g_0, g_1, g_2)$, encoding the genus $g$ of the central surface $\Sigma$ and the genera $g_0,$ $g_1,$ and  $g_2$ of the three $4$-dimensional handlebodies. The \emph{trisection genus} of $M,$ denoted $g(M),$ is the minimal genus of a central surface in any trisection of $M.$ A trisection with $g(\Sigma) = g(M)$ is called a \emph{minimal genus trisection}.

For example, the standard $4$-sphere has trisection genus equal to zero, the complex projective plane equal to one, and $S^2\times S^2$ equal to two. Meier and Zupan~\cite{MZ-standard} showed that there are only six orientable $4$-manifolds of trisection genus at most two. An extended set of examples of trisections of $4$-manifolds can be found in \cite{GK}. The recent works of Gay \cite{Gay2015}, and Meier, Schirmer and Zupan~\cite{MSZ, MZ-bridge, MZ-standard} give some applications and constructions arising from trisections of $4$-manifolds and relate them to other structures on $4$-manifolds.

A key feature is the so-called \emph{trisection diagram} of the $4$-manifold, comprising three sets of simple closed curves on the $2$-dimensional central surface from which the $4$-manifold can be reconstructed and various invariants of the $4$-manifold can be computed. This is particularly interesting in the case where the central surface is of minimal genus, giving a minimal representation of the $4$-manifold.

An approach to trisections using triangulations is given by Rubinstein and the second author in~\cite{Rubinstein-multisections-2016}, and used by Bell, Hass, Rubinstein and the second author in~\cite{Bell-computing-2017} to give an algorithm to compute trisection diagrams of $4$-manifolds, using a modified version of the input triangulation. The approach uses certain partitions of the vertex set (called \emph{tricolourings}), and the resulting trisections are said to be \emph{supported} by the triangulation.
We combine this framework with a greedy-type algorithm for collapsibility to explicitly calculate trisections from existing triangulations of $4$-manifolds. In doing so we are able to prove the following statement about the \emph{K3 surface}, a $4$-manifold that can be described as a quartic in $\mathbb{C}P^3$ given by the equation
\[ z_0^4 + z_1^4 + z_2^4 + z_3^4 =0.\]

\begin{theorem}
  \label{thm:main}
  The trisection genus of the $K3$ surface is $22$, that is, it is equal to its second Betti number.
\end{theorem}
The proof of the theorem consists of combining a theoretical lower bound on trisection genus in terms of Betti numbers (see Lemma~\ref{lem:g}) with an explicit computation of a trisection realising this lower bound.
The algorithms of \cite{Bell-computing-2017} can in particular be applied to a minimal trisection of the $K3$ surface to compute a minimal trisection diagram. This straightforward application is not done in this paper. Instead we focus on the following application of our result.

We already know that the trisection genera of $\mathbb{C}P^2$ and $S^2 \times S^2$ are equal to their respective second Betti numbers. Moreover, the second Betti number is additive, and the trisection genus is subadditive under taking connected sums. It thus follows from Theorem~\ref{thm:main} and Lemma~\ref{lem:g} that the trisection genus of every $4$-manifold which is a connected sum of arbitrarily many (PL standard) copies of $\mathbb{C}P^2$, $S^2 \times S^2$ and the $K3$ surface must be equal to its second Betti number. This family includes the standard $4$-sphere as the empty connected sum.

We refer to each member of this family of manifolds as a {\em standard simply connected PL $4$-manifold}.
Due to work by Freedman \cite{Freedman82Top4DimMnf}, Milnor and Husemoller \cite{Milnor73SymmBilForms}, Donaldson \cite{Donaldson83GaugeTheory4Mflds}, Rohlin \cite{Rohlin84NewResults4Mflds}, and Furuta \cite{Furuta01MonopoleEq}, each of these manifolds must be homeomorphic to one of
$$ k (\mathbb{C}P^2) \# m (\overline{\mathbb{C}P^2}) , \,\,  k (K3) \# m (S^2 \times S^2) , \textrm{ or }  k (\overline{K3}) \# m (S^2 \times S^2) , \quad  k,l,m,r \geq 0,$$ where $\overline{X}$ denotes $X$ with the opposite orientation.
 Furthermore, modulo the $11/8$-conjecture, standard simply connected PL $4$-manifolds comprise all topological types of PL simply connected $4$-manifolds. See \cite[Section~5]{Saveliev12Lectures} for a more detailed discussion of the above statements, and see Section~\ref{ssec:mfds} for more details about $4$-manifolds.

By analysing a particular family of singular triangulations of $\mathbb{C}P^2$, $S^2 \times S^2$ and the $K3$ surface due to Basak and the first author \cite{BaSCryst}, we are able to prove the following even stronger statement.

\begin{theorem}
  \label{thm:2}
  Let $M$ be a standard simply connected PL $4$-manifold. Then $M$ has a trisection of minimal genus supported by a singular triangulation of~$M$.
\end{theorem}

The singular triangulations used as building blocks (connected summands) for this construction are all highly regular and come from what is called {\em simple crystallisations} (see Section~\ref{ssec:simple} for details). They are explicitly listed in Appendix~\ref{app:isosigs}.

\medskip

\textbf{Experiments on orientable 4-manifolds.} Complementing the theoretical part of the article, we run a variant of the algorithm from~\cite{Bell-computing-2017} on the $440\,495$ triangulations of the $6$-pentachoron census of closed orientable $4$-manifolds~\cite{4d-census}. The algorithm takes as input a triangulation and searches for trisections which can be directly constructed from the simplicial structure of the triangulation. Unlike in the algorithm from~\cite{Bell-computing-2017} the structure of the triangulation is not altered. Accordingly, if trisections can be constructed this way, they are typically of small genus. 

Surprisingly, while only $445$ of these triangulations admit the construction of at least one trisection supported by its simplicial structure, some triangulations from the census admit as many as $48$ of them.

\medskip

\textbf{Experiments on non-orientable 4-manifolds.} The definition and existence result of trisections for PL 4--manifolds due to Rubinstein and the second author in~\cite{Rubinstein-multisections-2016} does not require the 4--manifold to be orientable. Likewise, our algorithm does not require the triangulation to be orientable and hence is able to construct trisections of non-orientable $4$-manifolds (see Definition~\ref{def:multisection} which defines trisections for both oriented and non-orientable $4$-manifolds). Going through the $60\,413$ triangulations of the $6$-pentachoron census of closed non-orientable $4$-manifolds~\cite{4d-census} we describe $13$ triangulations -- four of $S^3 \tilde{\times} S^1$ and nine of $\mathbb{R}P^4$ -- admitting at least one tricolouring on their sets of vertices without mono-chromatic triangles. Among these $13$ {\em tricolourable} triangulations, four admit the construction of a trisection compatible with their simplicial structure: two triangulations of $S^3 \tilde{\times} S^1$ give rise to a trisection of type $(2;1,1,1)$ and two triangulations of $\mathbb{R}P^4$ give rise to a trisection of type $(4;1,1,1)$. In both cases the trisections can be shown to be of minimal genus and hence determine the genus of the respective $4$-manifolds, again appealing to a theoretical lower bound in terms of Betti numbers (Lemma~\ref{lem:g}) and computing an explicit trisection that achieves this lower bound.

It follows from~\cite{Bell-computing-2017} that applying bistellar $2$-$4$-moves\footnote{A bistellar $2$-$4$-move replaces a pair of pentachora glued along a common tetrahedron by a collection of four pentachora glued around a common edge. Accordingly, each $2$-$4$-move increases the number of pentachora in the triangulation by two.} to a tricolourable triangulation eventually yields a triangulation such that its simplicial structure is compatible with a trisection. We apply this strategy to construct trisections from of all remaining nine tricolourable census-triangulations. Some of the trisections obtained in this way are of minimal genus.

\medskip

\textbf{Overview.} The paper is organised as follows. In Section~\ref{sec:prelims} we briefly go over some basic concepts used in the article, some elementary facts on trisections, as well as how to construct trisections from triangulations of $4$-manifolds. In Section~\ref{sec:implementation} we give details on the algorithm to compute trisections from a given triangulation and we present data obtained from running the algorithm on the $6$-pentachora census of closed (orientable and non-orientable) $4$-manifolds. In Section~\ref{sec:simply} we then prove Theorems~\ref{thm:main} and~\ref{thm:2}.

\subsection*{Acknowledgements}

Research of the first author was supported by the Einstein Foundation (project ``Einstein Visiting Fellow Santos'').
Research of the second author was supported in part under the Australian Research Council's Discovery funding scheme (project number DP160104502). The second author thanks the DFG Collaborative Center SFB/TRR 109 at TU Berlin, where parts of this work have been carried out, for its hospitality.


\section{Preliminaries}
\label{sec:prelims}

\subsection{Manifolds and triangulations}
\label{ssec:mfds}

We assume basic knowledge of geometric topology and in particular manifolds. For a definition of homology groups and Betti numbers of a manifold see \cite{Hatcher2002AlgTop}, for an introduction into simply connected $4$-manifolds and their intersection forms, see \cite{Saveliev12Lectures}.

In dimensions $\le 4$, there is a bijective correspondence between isotopy classes of smooth and piecewise linear structures~\cite{C1, C2}. In this paper, all manifolds and maps are assumed to be piecewise linear (PL) unless stated otherwise.
Our results apply to any compact smooth manifold by passing to its unique piecewise linear structure~\cite{Whitehead1940}.

Recall that a manifold $M$ is {\em simply connected} if every closed loop in $M$ can be contracted to a single point (in other words, the fundamental group of $M$ is trivial). While in dimensions $2$ and $3$ the only simply connected closed manifolds are the $2$- and $3$-sphere respectively, in dimension $4$ there are infinitely many such manifolds.

Given two oriented simply connected $4$-manifolds $M$ and $N$, their {\em connected sum}, written $M \# N$, is defined as follows. First remove a $4$-ball from both $M$ and $N$ and then glue the resulting manifolds along their boundaries via an orientation reversing homeomorphism. The resulting manifold is again oriented. The second Betti number is additive and the property of being simply connected is preserved under taking connected sums.

In this article, simply connected $4$-manifolds are presented in the form of {\em singular triangulations}.
A singular triangulation $\tri$ of a (simply connected) $4$-manifold $M$
is a collection of $2n$ abstract pentachora together with $5n$ {\em gluing maps} identifying their $10n$ tetrahedral faces
in pairs such that the underlying set of the quotient space $|\tri|$ is PL-homeomorphic to $M$.
The gluing maps generate an equivalence relation on the faces of the pentachora, and an equivalence class of faces is referred to as a single face of the triangulation $\tri$. Faces of dimension zero are called {\em vertices} and faces of dimension one are referred to as {\em edges} or the triangulation. The triangulation is \emph{simplicial} if each equivalence class of faces has at most one member from each pentachoron; i.e. no two faces of a pentachoron are identified in $|\tri|.$ Every simplicial triangulation is also singular, and the second barycentric subdivision of a singular triangulation is always simplicial.
The triangulation is \emph{PL} if, in addition, the boundary of a small neighbourhood of every vertex is PL homeomorphic to the $3$-sphere. The number $2n$ of pentachora of $\tri$ is referred to as the size of~$\tri$.

\subsection{Trisections}

We begin with a formal definition of a trisection of a $4$-manifold.

\begin{definition}[Trisection of closed manifold]
\label{def:multisection}
Let $M$ be a closed, connected, piecewise linear $4$-manifold.
A \emph{trisection} of $M$ is a collection of three piecewise linear codimension zero submanifolds $H_0, H_1, H_2 \subset M$, subject to the following four conditions:
\begin{enumerate}
\setlength\itemsep{0em}
\item  Each $H_i$  is PL homeomorphic to a standard piecewise linear $4$-dimensional $1$-handlebody of genus $g_i$.
\item The handlebodies $H_i$ have pairwise disjoint interior, and $M = \bigcup_i H_i$.
\item The intersection $H_{i} \cap H_{j}$ of any two of the handlebodies is a $3$-dimensional $1$-handlebody.
\item The common intersection $\Sigma = H_{0} \cap H_{1} \cap H_{2}$ of all three handlebodies is a closed, connected surface,  the \emph{central surface}.
\end{enumerate}
\end{definition}
The submanifolds   $H_{ij} = H_{i} \cap H_{j}$  and $\Sigma$ are referred to as the \emph{trisection submanifolds}.
In our illustrations, we use the colours blue, red, and green instead of the labels $0$, $1$, and $2$ and we refer to $H_{\textrm{blue}\thinspace \textrm{red}} = H_{br}$ as the \emph{green} submanifold and so on.

The trisection in the above definition is also termed a $(g; g_0, g_1, g_2)$-trisection, where $g=g(\Sigma)$ is the genus of the central surface if $\Sigma$ is orientable, and the {\em non-orientable genus}, i.e., the number $k$ of crosscaps $\mathbb{R}P^2$ such that $\Sigma = (\mathbb{R}P^2)^{\# k}$, if $\Sigma$ is non-orientable. 

Note that in the original definition given by Gay and Kirby~\cite{GK} they ask for the trisection to be \emph{balanced} in the sense that each handlebody $H_i$ has the same genus.
It is noted in \cite{Meier-classification-2016, Rubinstein-multisections-2016} that any unbalanced trisection can be stabilised to a balanced one.\footnote{A stabilisation of a trisection with central surface of genus g is obtained by attaching a 1-handle to a 4-dimensional handlebody along a properly embedded boundary parallel arc in the 3-dimensional handlebody that is the intersection of the other two 4-dimensional handlebodies of the trisection. The result is a new trisection with central surface of genus g + 1, and exactly one of the 4-dimensional handlebodies has its genus increased by one.}	

\begin{definition}[Trisection genus]
  \label{def:genus}
  The \emph{trisection genus} of $M,$ denoted $g(M),$ is the minimal genus of a central surface in any trisection of $M.$ A trisection with $g(\Sigma) = g(M)$ is called a \emph{minimal genus trisection}.
\end{definition}

Definition~\ref{def:multisection} and hence also Definition~\ref{def:genus} are -- in their original form -- restricted to oriented $4$-manifolds~\cite{GK}. This restriction is not necessary for the purpose of the definition (see \cite{Rubinstein-multisections-2016}), but applying the definition to both orientable and non-orientable manifolds has the following immediate consequence.

\begin{proposition}
  \label{prop:nonor}
  Let $M$ be a closed, connected, piecewise linear $4$-manifold. Then $M$ is orientable if and only if all parts in all trisections of $M$ are orientable, and $M$ is non-orientable if and only if all parts in all trisections of $M$ are non-orientable.
\end{proposition}

\begin{proof}
  Assume $M$ has a trisection with $4$-dimensional handlebodies $H_0$, $H_1$ and $H_2$, $3$-dimensional handlebodies $H_{01}$, $H_{02}$ and $H_{12}$, and central surface $\Sigma$.

  Assume that one of the $4$-dimensional handlebodies, say $H_0$, is orientable. Then $\partial H_0 = \partial (H_1 \cup H_2)$ is orientable. Since decomposing an orientable $3$-manifold into two $3$-dimensional handlebodies along their common and connected boundary must result in two orientable handlebodies, it follows that $H_{01}$, $H_{02}$ and $\Sigma$ must be orientable. But $\Sigma = \partial H_{12}$ is orientable and thus $H_{12}$ as well because it is a handlebody. Finally, a decomposition of a $3$-manifold into two orientable handlebodies along their connected boundaries must be an orientable manifold and it follows that $\partial H_1$ and $\partial H_2$ must be orientable and hence $H_1$ and $H_2$ as well. 

  If one of the $3$-dimensional handlebodies is orientable, then $\Sigma$ is orientable and hence all other three $3$-dimensional handlebodies are orientable. If follows that they pairwise constitute a Heegaard splitting of the orientable boundaries of the three $4$-dimensional handlebodies and thus all seven parts of the trisection must be orientable. The same argument holds if $\Sigma$ is assumed to be orientable.

  In particular, either all seven parts of a trisection are orientable or all seven parts are non-orientable. We conclude the proof by noting that, on the one hand, an orientation on $H_0$ extends to an orientation on $H_1$ and $H_2$ and $M$ is orientable and, on the other hand, $H_0$ being non-orientable directly implies that $M$ must be non-orientable.
\end{proof}

It follows that, in a $(g; g_0, g_1, g_2)$-trisection of a non-orientable $4$-manifold $M$, we must have $g_i \geq 1$, $1 \leq i \leq 3$, and $g$ even.

\subsection{Three elementary facts}

\begin{lemma}
\label{lem:000}
If $M$ has a trisection in which all $4$-dimensional handlebodies are $4$-balls, then the trisection is a minimal genus trisection.
\end{lemma}

\begin{proof}
Suppose $M$ has a $(g'; 0, 0, 0)$-trisection. In particular, the fundamental group of $M$ is trivial, $M$ is orientable, and the trisection is a trisection in the sense of \cite{GK}. Moreover, suppose that $M$ has also a $(g; g_0, g_1, g_2)$-trisection and $g = g(M).$ By \cite{GK}, these have a common stabilisation. Suppose this is a $(g''; k_0, k_1, k_2)$-trisection. Each elementary stabilisation increases the genus of one handlebody and the genus of the central surface by one. Hence $g'' = g' + k_0+ k_1+ k_2$ and $g'' = g + (k_0 - g_0) + (k_1-g_1) + (k_2-g_2).$ This gives $g' \ge g (M) = g = g' + g_0+ g_1+ g_2.$ This forces $g_0 = g_1 = g_2=0$ and $g=g'.$
\end{proof}

\begin{lemma}
\label{lem:g}
Suppose $M$ has a $(g; g_0, g_1, g_2)$-trisection. Then, $g \ge \beta_1 (M, \mathbb{Z}_2) + \beta_2 (M, \mathbb{Z}_2)$ if $M$ is orientable, and $g \ge 2 (\beta_1 (M, \mathbb{Z}_2) + \beta_2 (M, \mathbb{Z}_2))$ if $M$ is non-orientable (and $g$ denotes the non-orientable genus).
\end{lemma}

\begin{proof}
Let $H_0$, $H_1$ and $H_2$ denote the $4$-dimensional handlebodies, and let $\Sigma$ be the central surface of the trisection. A simple calculation shows that $\chi(M) = 2 + (1-\chi(\Sigma)/2) - g_0 -g_1 - g_2$. On the other hand, since the fundamental group of $H_i$, $0 \leq i\leq 2$, surjects onto the fundamental group of $M$, we have that $\beta_1 (M, \mathbb{F}) \leq g_i$ for all $0 \leq i \leq 2 $ and for all fields of coefficients. Moreover, since all manifolds are orientable with respect to the field with two elements $\mathbb{Z}_2$, it follows from Poincar\'e duality that $\beta_1 (M, \mathbb{Z}_2) = \beta_3 (M, \mathbb{Z}_2)$.  

Combining all of the identities and inequalities above finishes the proof.
\end{proof}

The orientable case of Lemma~\ref{lem:g} is contained in the introduction of \cite{Chu18TrisectionGenus} due to Chu and the second author. Note that Lemma~\ref{lem:g} can also be followed from Equation (3.9) in \cite{Feller-calculating-2017}, where it is shown how to compute the homology of a $4$-manifold from a trisection diagram.

\begin{lemma}
  \label{lem:subadditive}
  If $M_0$ has a trisection with central surface of Euler characteristic $\chi_0$ and $M_1$ has a trisection with central surface of Euler characteristic $\chi_1,$ then the connected sum of $M_0$ and $M_1$ has a trisection with central surface of Euler characteristic $\chi_0 + \chi_1 - 2.$ In particular, $g(M_0 \# M_1) \leq g(M_0) + g(M_1)$ if $M_0$ and $M_1$ are either both orientable or both non-orientable and $g(M_0 \# M_1) \leq 2 g(M_0) + g(M_1)$ if $M_0$ is orientable and $M_1$ is non-orientable.
\end{lemma}

\begin{proof}
  If both manifolds $M_0$ and $M_1$ are orientable and oriented, then there is a well-defined connected sum as explained in \cite[\S2]{GK}. This follows by choosing, for the connect sum operation, two standardly trisected $4$-balls in both $4$-manifolds. The proof is finished by substituting $\chi = 2 - 2g$ for the (orientable) central surfaces of the respective trisections.

  If both $M_0$ and $M_1$ are non-orientable, then we can use local orientations to define the connected sum, and the result follows from the formula $\chi = 2 - g$ for the genus of non-orientable surfaces.

  If $M_0$ is orientable and $M_1$ is non-orientable, then we again use local orientations to perform the connected sum. Since $M_0 \# M_1$ is non-orientable in this case, its trisection genus now relates to the non-orientable genus of the central surface in a minimal genus trisection and thus the formula for the genus must be adjusted accordingly.
\end{proof}

\subsection{Algorithms to compute trisections}

Our set-up for algorithms to compute trisections follows \cite{Bell-computing-2017}, where a trisection on $M$ is induced by \emph{tricolourings} of the triangulation. This is now summarised; the reader is referred to \cite{Bell-computing-2017} for a more detailed discussion.

\begin{figure*}
\centering
\includegraphics[width=.8\linewidth]{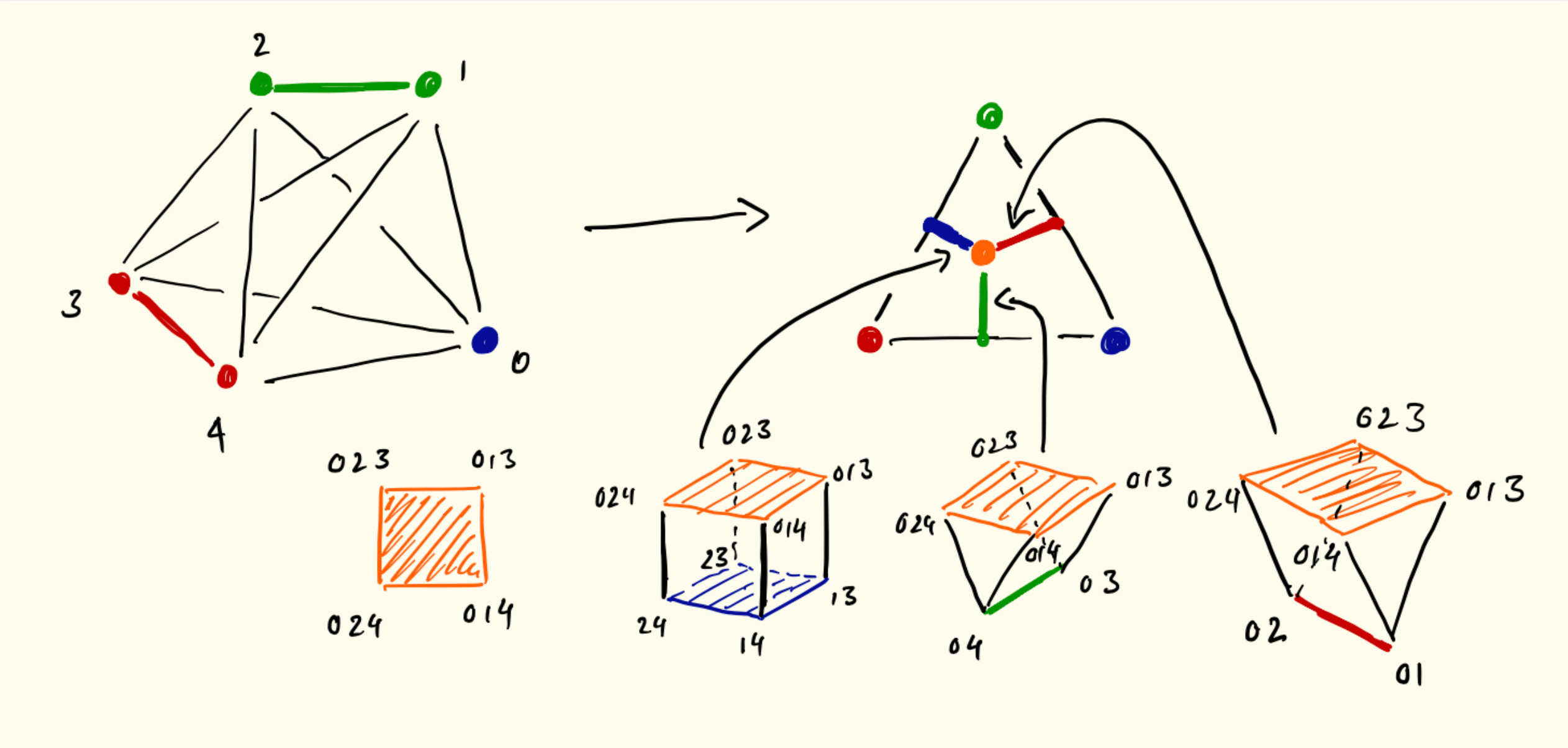}
\caption{Pieces of the trisection submanifolds. The vertices of the pieces are barycentres of faces and labelled with the corresponding vertex labels. The central surface meets the pentachoron in a square.
Two of the $3$-dimensional trisection submanifolds meet the pentachoron in triangular prisms and the third (corresponding to the singleton) meets it in a cube. Moreover, any two of these meet in the square of the central surface.}
\label{fig:pieces}
\end{figure*}

Let $M$ be a closed, connected $4$-manifold with (not necessarily simplicial, but possibly singular) triangulation $\tri$.
A partition $\{P_0, P_1, P_2\}$ of the set of all vertices of $\tri$ is a \emph{tricolouring} if every pentachoron meets two of the partition sets in two vertices and the remaining partition set in a single vertex. In this case, we also say that the triangulation is \emph{tricoloured}. Moreover, we call a triangulation {\em tricolourable} if it admits at least one tricolouring.

Denote the vertices of the standard $2$-simplex $\Delta^2$ by $v_0$, $v_1$, and $v_2$.
A tricolouring determines a natural map $\mu \co M \to \Delta^2$ by sending the vertices in $P_k$ to $v_k$ and extending this map linearly over each simplex.
Note that the pre-image of $v_k$ is a graph $\Gamma_k$ in the $1$-skeleton of $M$ spanned by the vertices in~$P_k$.

The strategy in \cite{Bell-computing-2017,Rubinstein-multisections-2016} is to use $\mu$ to pull back the \emph{cubical structure} of the simplex to a trisection of $M$.
The \emph{dual spine} $\Pi^n$ in an $n$-simplex $\Delta^n$ is the $(n-1)$-dimensional subcomplex of the first barycentric subdivision of $\Delta^n$ spanned by those vertices of the first barycentric subdivision which are not vertices of $\Delta^n$ itself.
This is shown for $n=2$ and $n=3$ in Figure~\ref{fig:Pi3}.
Decomposing along $\Pi^n$ gives $\Delta^n$ a natural
 \emph{cubical structure} with $n+1$ cubes of dimension $n$, and the lower-dimensional cubes that we focus on are the intersections of non-empty collections of these top-dimensional cubes.

\begin{figure}[h]
    \centering
    \includegraphics[height=3.4cm]{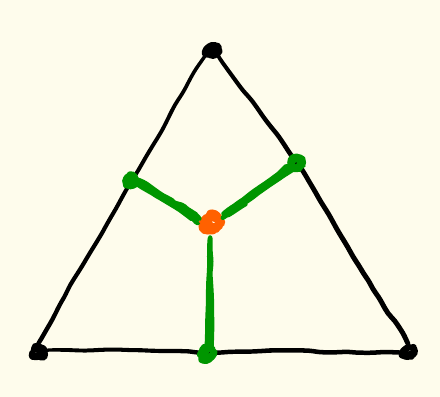}\hspace{2.5cm}
    \includegraphics[height=3.4cm]{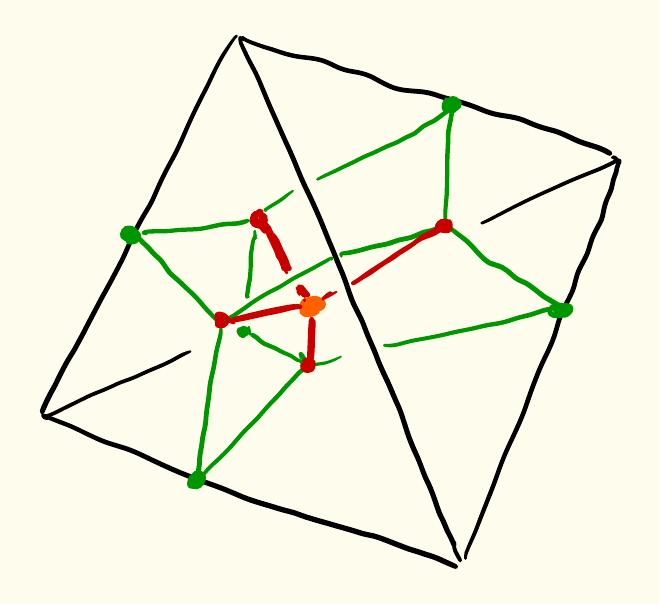}
\caption{Dual cubes. Left: $\Pi^2\subset \Delta^2$. Right: $\Pi^3\subset \Delta^3$.\label{fig:Pi3}}
\end{figure}

Recall that a compact subpolyhedron $P$ in the interior of a manifold $M$ is called a \emph{(PL) spine of $M$} if $M$ collapses onto $P.$ If $P$ is a spine of $M$, then $M \setminus P$ is PL homeomorphic with $\partial M \times [0,1).$

The pre-images under $\mu$ of the dual cubes of $\Pi^2 \subset \Delta^2$ have very simple combinatorics.
The barycentre of $\Delta^2$ pulls back to exactly one $2$-cube in each pentachoron of $M$, and these glue together to form a surface $\Sigma$ in $M$.
This surface is the common boundary of each of the three $3$-manifolds obtained as pre-images of an interior $1$-cube (edge) of $\Delta^2$.
Each such $3$-manifold is made up of cubes and triangular prisms, as in Figure~\ref{fig:pieces}.
Each interior $1$-cube $c$ has boundary the union of the barycentre of $\Delta^2$ and the barycentre $b$ of an edge of $\Delta^2$.
Since the map $\mu \co M \to \Delta^2$ is linear on each simplex, the pre-image $\mu^{-1}(c)$ collapses to the pre-image $\mu^{-1}(b)$.
In particular, each $3$-manifold has a spine made up of $1$-cubes and $2$-cubes.

It is shown in \cite{Bell-computing-2017} that the above construction gives a trisection if:
\begin{enumerate}
\setlength\itemsep{0em}
\item the graph $\Gamma_k$ is connected for each $k$; and
\item the pre-image of an interior $1$-cube of $\Delta^2$ has a $1$-dimensional spine.
\end{enumerate}

A tricolouring is a \emph{c-tricolouring} if $\Gamma_k$ is connected for each $k.$ A c-tricolouring is a \emph{ts-tricolouring} if the pre-image of each interior $1$-cube collapses onto a $1$-dimensional spine. In this case, the dual cubical structure of $\Delta^2$ pulls back to a trisection of~$M$.

\begin{definition}[Trisection supported by triangulation]
We say that a trisection of $M$ is {\em supported} by the triangulation $\tri$ of $M$, if $\tri$ is ts-tricolourable and the trisection is isotopic to the pull-back of the dual cubical structure of $\Delta^2.$ In this case, the trisection is said to be \emph{dual} to the corresponding ts-tricolouring.
\end{definition}

An algorithm to construct a ts-tricolouring from an arbitrary initial triangulation is given in \cite{Bell-computing-2017}. It uses subdivisions and bistellar moves in order to justify that the resulting triangulation has a ts-tricolouring. As a result, the number of pentachora of the final triangulation is larger by a factor of 120 compared to the number of pentachora of the initial one (see \cite[Theorem 4]{Bell-computing-2017}.)

The new algorithmic contribution of this paper is the development and implementation of an algorithm to determine whether a given triangulation directly admits a ts-tricolouring. This is given in \S\ref{sec:implementation}.

\subsection{Crystallisations of simply connected $4$-manifolds}
\label{ssec:simple}

Let $G = (V, E)$ be a multigraph without loops. An {\em edge colouring} of
$G$ is a surjective map $\gamma \co E \to C = \{0, 1, \ldots , d \}$
such that $\gamma(e) \neq \gamma(f)$ whenever $e$ and $f$ are adjacent.
The tuple $(G, \gamma)$ is said to be a {\em $(d+1)$-coloured graph}.
For the remainder of this section we fix $d$ to be~$4$.

For a $5$-coloured graph $(G,\gamma)$, and $B \subseteq C$,
$|B| = k$, the graph $G_B = (V(G), \gamma^{-1}(B))$ together with the colouring $\gamma$ restricted to $\gamma^{-1}(B)$ is a $k$-coloured graph. If for all $j \in C$, $G_{C\setminus\{j\}}$
is connected, then $(G,\gamma)$ is called {\em contracted}.

A $5$-coloured graph $(G, \gamma)$ defines a $4$-dimensional simplicial cell complex $K(G)$:
For each $v\in V(G)$ take a pentachoron $\Delta_v$ and label
its vertices by $0, 1,2,3, 4$. If $u, v \in V(G)$ are joined by an edge $e$ in $G$ and $\gamma(e) =  j$,
we identify $\Delta_u$ and $\Delta_v$ along the tetrahedra opposite to vertex $j$ with
equally labelled vertices glued together. This way, no face of $K(G)$ can have self-identifications and
$K(G)$ is a regular simplicial cell complex. We say that $(G, \gamma)$
{\em represents} the simplicial cell complex $K(G)$.
Since, in addition, the number of $j$-labelled vertices of $K(G)$ is equal to the number of components of
$G_{C\setminus\{j\}}$ for each $ j\in C$, the simplicial cell complex $K(G)$ contains exactly $5$ vertices
if and only if $G$ is contracted~\cite{Ferri86GraphTheoryCrystallizations}.

For a $4$-manifold $M$ we call a $5$-coloured contracted graph $(G, \gamma)$ a {\em crystallisation} of
$M$ if the simplicial cell complex $K(G)$ is a singular triangulation of $M$ (with no self-identifications of any of its faces).
Every PL $4$-manifold admits such a crystallisation due to work by Pezzana~\cite{Pezzana74Crystallizations}.

We refer to $(G, \gamma)$ as simple, if $K(G)$ has exactly $10$ edges. That is, if the one-skeleton of $K(G)$ is equal
to the $1$-skeleton of any of its pentachora. While not every simply connected PL $4$-manifold can admit a simple crystallisation\footnote{Note that, on the one hand there exist simply connected topological $4$-manifolds with an infinite number of PL structures, and on the other hand, every topological type of simply connected PL $4$-manifold can only admit a finite number of simple crystallisations.},
this is true for the standard ones, see the article by Basak and the first author \cite{BaSCryst} for an explicit construction.


\section{Implementation and data}
\label{sec:implementation}

In this section we describe the algorithm to check whether a singular triangulation admits a trisection and present some experimental data for the Budney-Burton census of orientable and non-orientable singular triangulations of $4$-manifolds up to six pentachora \cite{4d-census}.

\subsection{Implementation}
\label{ssec:algo}

Given a singular triangulation $\tri$ with vertex set $V=V(\tri)$, perform the following three basic and preliminary checks.

\begin{enumerate}
  \item Check if $\tri$ has at least three vertices.
  \item For all triangles $t \in \tri$, check if $t$ contains at least two vertices.
  \item For all partitions of the vertex set into three non-empty parts $ P_0 \sqcup P_1 \sqcup P_2 = V$, for all triangles $t \in \tri$, check that no triangle is monochromatic with respect to $(P_0,P_1,P_2)$ (i.e., every triangle contains vertices from at least two of the $P_i$, $i=0,1,2$).
\end{enumerate}

If any of these checks fail we conclude that $\tri$ does not admit a valid tricolouring (i.e., $\tri$ is not tricolourable). Otherwise, for each tricolouring of $\tri$ we proceed in the following way.

Compute the spines of the $4$-dimensional handlebodies, that is, the graphs $\Gamma_i$ in the $1$-skeleton of $\tri$ spanned by the vertices in $P_i$, $i=0,1,2$. If all of the $\Gamma_i$, $i=0,1,2$, are connected, $(P_0,P_1,P_2)$ defines a c-tricolouring of $\tri$. Otherwise we conclude that the partition $(P_0,P_1,P_2)$ is not a c-tricolouring of~$\tri$.

A spine of the $3$-dimensional trisection submanifolds is given by the $1$- and $2$-cubes sitting in bi-coloured triangles and tetrahedra. Denote these three complexes $\gamma_i$, $i=0,1,2$ (with $\gamma_i$ being disjoint of all faces containing vertices in $P_i$). We compute the Hasse diagram of each of the $\gamma_i$ and check their connectedness. In case all three complexes are connected, we perform a greedy-type collapse (as defined in \cite{Benedetti13RandomDMT}) to check whether $\gamma_i$ collapses onto a $1$-dimensional complex (and thus the $3$-dimensional trisection submanifolds are handlebodies). Note that this is a deterministic procedure, see, for instance, \cite{Tancer12CollNPComplete}.

If all of these checks are successful, $(P_0,P_1,P_2)$ defines a ts-tricolouring and thus a trisection of $M$ dual to this ts-tricolouring of $\tri$. Otherwise we conclude that $(P_0,P_1,P_2)$ does not define a ts-tricolouring.

Finally we compute the central surface, its genus, the genera of the handlebodies $\Gamma_i$ and $\gamma_i$, $i=0,1,2$, and, as a plausibility check, compare the genus of the central surface to the genera of the $3$-dimensional handlebodies.

\smallskip

For $\tri$ a $v$-vertex, $n$-pentachora triangulation the running time of this procedure is roughly the partition number $p(v)$ times a small polynomial in $n$. Since the triangulations we need to consider usually come with a very small number of vertices, the running time of our algorithm is sufficiently feasible for our purposes.

\subsection{Experimental data}
\label{ssec:expts}

\paragraph{The orientable census}

The census of singular triangulations of orientable, closed $4$-manifolds \cite{4d-census} contains $8$ triangulations with two pentachora, $784$ triangulations with four pentachora and $440\,495$ triangulations with $6$ pentachora.

In the $2$-pentachora case, exactly $6$ of the $8$ triangulations have at least $3$ vertices, and in exactly $3$ of them all triangles contain at least two vertices. Overall, these three triangulations admit $19$ tricolourings, $18$ of which are c-tricolourings covering $2$ triangulations. Of these $18$ c-tricolourings all are dual to trisections. Overall, exactly $2$, that is, $25\%$ of all $2$-pentachora triangulations admit a trisection.

Of the $784$ triangulations with $4$-pentachora, exactly $324$ have at least three vertices, and exactly $24$ only have triangles containing at least two vertices. Each of these $24$ triangulations admits at least one tricolouring. Altogether they admit a total of $88$ tricolourings. In $15$ triangulations, at least one tricolouring defines connected graphs $\Gamma_i$, $0,1,2$, and is thus a c-tricolouring. These $15$ triangulations admit a total of $72$ c-tricolourings. All $72$ of them are dual to trisections giving rise to a total of $15$ triangulations ($1.9\%$ of the census) in the $4$-pentachora census admitting a trisection.

There are $440\,495$ triangulations in the $6$-pentachoron census, exactly $116\,336$ of which have at least three vertices, and exactly $837$ of which have no triangles with only one vertex. $824$ of these $837$ triangulations admit at least one tricolouring, summing up to a total of $1\,689$ tricolourings in the census. Amongst these, $450$ triangulations have at least one $c$-tricolouring, summing up to a total of $1\,100$ c-tricolourings. In $8$ of these $1\,100$ c-tricolourings, the $2$-complexes $\gamma_i$, $i=0,1,2$, are not all connected, and in a further $5$ of them not all $\gamma_i$, $i=0,1,2$, collapse to a $1$-dimensional complex. The remaining $1,087$ ts-tricolourings occur in $445$ distinct triangulations. It follows that a total of $445$ (or $0.1\%$) of all $6$-pentachora triangulations admit a trisection.

Amongst the $445$ triangulations supporting at least one trisection, there are $11$ triangulations, nine of the $4$-dimensional sphere $S^4$ and two of $S^3 \times S^1$, supporting a trisection which is of minimal genus (e.g., due to Lemma~\ref{lem:g}). In particular, this demonstrates the well-known facts that $g(S^4)=0$ and $g(S^3 \times S^1)=1$. See Table~\ref{fig:minimal} for a list of all eight triangulations.

\begin{table}[thbp]
 \begin{center}
  \caption{$6$-pentachora census-triangulations supporting a trisection of minimal genus. \label{fig:minimal}}
  \begin{tabular}{l|c|c}
    \toprule
    iso. sig. & top. type & trisection \\
    \midrule
\texttt{gLMPMQccdeeeffffaaaa9aaaaaaaaaaaaa9a} & $S^3 \times S^1$ &$ (1; 1, 1, 1)$ \\
\texttt{gLwMQQcceeeffeffaaaaaaaaaaLaLaLaLaLa} & \dittostraight &\dittostraight \\
\texttt{gLAAMQbbcddeffffaaaaaaaaaaaaaaaaaaaa} & $S^4$ &$ (0; 0, 0, 0)$ \\
\texttt{gLAMMQbccddeffffaaaaaaaaaaaaaaaaaaaa} & \dittostraight &\dittostraight \\
\texttt{gLwPQQbbeeedffffaaaaaaaaaavaaaaaaava} & \dittostraight &\dittostraight \\
\texttt{gLAAMQbccddeffffaaaaaaaaaaaaaaaaaaaa} & \dittostraight &\dittostraight \\
\texttt{gvLQQQcdefdefeffya2aqbPbgaoavacafaba} & \dittostraight &\dittostraight \\
\texttt{gLAAMQbccdddffffaaaaaaaaaaaaaaaaaaaa} & \dittostraight &\dittostraight \\
\texttt{gLAPMQbbbeeeffffaaaaaaaaaaaaaaaaaaaa} & \dittostraight &\dittostraight \\
\texttt{gLAAMQbbbdddffffaaaaaaaaaaaaaaaaaaaa} & \dittostraight &\dittostraight \\
\texttt{gLAAMQbbbddeffffaaaaaaaaaaaaaaaaaaaa} & \dittostraight &\dittostraight \\
    \bottomrule
  \end{tabular}
 \end{center}
\end{table}

Table~\ref{fig:data} gives an overview of how many tricolourings, c-tricolourings, and ts-tricolourings can be expected from a triangulation admitting at least one tricolouring. As can be observed from the table, there exist triangulations in the census with as many as $51$ distinct tricolourings, and $48$ distinct ts-tricolourings.

For many of the triangulations with multiple ts-tricolourings, all trisections dual to them have central surfaces of the same genus and $4$-dimensional handlebodies of the same genera. For some, however, a smallest genus trisection occurs amongst a number of trisections of higher genera. For instance, one of the largest ranges of genera of central surfaces is exhibited by a triangulation with $15$ ts-tricolourings, supporting trisections of type $(0;0,0,0)$ ($\times 10$), $(1;1,0,0)$ ($\times 4$), and $(2;1,1,0)$  ($\times 1$). The triangulation is necessarily homeomorphic to $S^4$. Its isomorphism signature is given in Table~\ref{fig:minimal}, line $3$, or Appendix~\ref{app:isosigs}.

\begin{table}[hbt]
 \begin{center}
  \caption{Number of triangulations in the $6$-pentachora census of orientable closed $4$-manifolds having $k$ tricolourings ($\mathfrak{n}_{\operatorname{tc}}$ in second column), c-tricolourings ($\mathfrak{n}_{\operatorname{c}}$ in third column), and ts-tricolourings ($\mathfrak{n}_{\operatorname{ts}}$ in fourth column). \label{fig:data}}
  \begin{tabular}{r|rrr}
    \toprule
    $k$ & $\mathfrak{n}_{\operatorname{tc}}$ & $\mathfrak{n}_{\operatorname{c}}$ & $\mathfrak{n}_{\operatorname{ts}}$ \\
    \midrule
    $  1 $&$ 518 $&$ 248 $&$ 243 $ \\
    $  2 $&$ 155 $&$  76 $&$  79 $ \\
    $  3 $&$  67 $&$  56 $&$  55 $ \\
    $  4 $&$  24 $&$  36 $&$  34 $ \\
    $  5 $&$  22 $&$   0 $&$   0 $ \\
    $  6 $&$  15 $&$  16 $&$  16 $ \\
    $  8 $&$   4 $&$   6 $&$   6 $ \\
    $  9 $&$   4 $&$   0 $&$   0 $ \\
    $ 10 $&$   3 $&$   1 $&$   1 $ \\
    $ 11 $&$   1 $&$   0 $&$   0 $ \\
    $ 12 $&$   2 $&$   2 $&$   2 $ \\
    $ 15 $&$   4 $&$   4 $&$   5 $ \\
    $ 18 $&$   1 $&$   1 $&$   0 $ \\
    $ 24 $&$   0 $&$   2 $&$   2 $ \\
    $ 27 $&$   2 $&$   0 $&$   0 $ \\
    $ 36 $&$   0 $&$   1 $&$   1 $ \\
    $ 48 $&$   1 $&$   1 $&$   1 $ \\
    $ 51 $&$   1 $&$   0 $&$   0 $ \\
    \midrule
    $\Sigma$&$824$&$ 450 $&$ 445 $ \\
    \bottomrule
    \bottomrule
  \end{tabular}
 \end{center}
\end{table}

\paragraph{The non-orientable census}

The algorithm presented in Section~\ref{ssec:algo} does not require the triangulated $4$-manifold to be orientable and hence is able to construct trisections of non-orientable $4$-manifolds as well. These are defined in the same way as in the oriented case (see Definition~\ref{def:multisection}) but have slightly different properties (see Proposition~\ref{prop:nonor}). Namely, all parts of the decomposition must be non-orientable and the central surface must have even non-orientable genus. 

Running the algorithm on the $60\,413$ triangulations of the $6$-pentachoron census of closed non-orientable $4$-manifolds~\cite{4d-census} there are very few triangulations supporting a valid tricolouring and even fewer (that is, only four) supporting a trisection (see Table~\ref{fig:datanonor}). Two of them are triangulations of the boundary of a genus one non-orientable handlebody $S^3 \tilde{\times} S^1$ and both support a trisection of type $(2;1,1,1)$. The other two are triangulations of the $4$-dimensional real projective space $\mathbb{R}P^4$, both supporting a trisection of type $(4;1,1,1)$, see also Table~\ref{fig:altered}.

\begin{table}[hbt]
 \begin{center}
  \caption{Number of triangulations in the $6$-pentachora census of non-orientable closed $4$-manifolds having $k$ tricolourings ($\mathfrak{n}_{\operatorname{tc}}$ in second column), c-tricolourings ($\mathfrak{n}_{\operatorname{c}}$ in third column), and ts-tricolourings ($\mathfrak{n}_{\operatorname{ts}}$ in fourth column). \label{fig:datanonor}}
  \begin{tabular}{r|rrr}
    \toprule
    $k$ & $\mathfrak{n}_{\operatorname{tc}}$ & $\mathfrak{n}_{\operatorname{c}}$ & $\mathfrak{n}_{\operatorname{ts}}$ \\
    \midrule

    $  1 $&$  12 $&$  11 $&$   4 $ \\
    $  2 $&$   1 $&$   1 $&$   0 $ \\

    \bottomrule
  \end{tabular}
 \end{center}
\end{table}


\begin{table}[thbp]
 \begin{center}
  \caption{Non-orientable tricolourable $6$-pentachora census-triangulations, their topological types, number of bistellar $2$-$4$-moves necessary for them to support a trisection, and minimum genus of supported trisection. \label{fig:altered}}
  \begin{tabular}{l|c|c|c}
    \toprule
    iso. sig. & top. type & $\#$ moves & min. genus \\
    \midrule
    \texttt{gLALQQbbbdefefffaaaaaaaaaa3b3b3b3b3b} & $\mathbb{R}P^4$ &$1 $ & $ (4;1,1,1) $ \\
    \texttt{gLwMQQcceeeffeffaaaaaaaaaa9a9a9a9a9a} & $S^3 \tilde{\times} S^1$ &$0 $ & $ (2;1,1,1) $ \\ 
    \texttt{gLMPMQccdeeeffffaaaa3aaaaaaaaaaaaa3a} & $S^3 \tilde{\times} S^1$ &$1 $ & $ (2;1,1,1) $ \\
    \texttt{gLMPMQccdeeeffffaaaaabaaaaaaaaaaaaab} & $S^3 \tilde{\times} S^1$ &$0 $ & $ (2;1,1,1) $ \\ 
    \texttt{gLLAQQcddcfdefffaaVbVbVbxbVbbaaaaaaa} & $\mathbb{R}P^4$ &$0 $ & $ (4;1,1,1) $ \\
    \texttt{gLLAQQbdedfefeefaadbdbdbaaoaaadbdboa} & $\mathbb{R}P^4$ &$1 $ & $ (6;2,1,1) $ \\
    \texttt{gLLAQQbdedffeeefaadbdbdbaaaavadbdbva} & $\mathbb{R}P^4$ &$1 $ & $ (6;2,1,1) $ \\ 
    \texttt{gLAAMQacbdcdefffcaTava4acavaya1aYa2a} & $S^3 \tilde{\times} S^1$ &$3 $ & $ (8;2,2,2) $ \\ 
    \texttt{gvLQQQcdefdfefefya2aqbPbgaGbjbSbvbba} & $\mathbb{R}P^4$ &$0 $ & $ (4;1,1,1) $ \\
    \texttt{gLALQQbbbeffefefaaaaaaaaaasbsbsbsbsb} & $\mathbb{R}P^4$ &$2 $ & $ (4;1,1,1) $ \\
    \texttt{gLAMPQbbcdeffeffaaaaaaaaaadbdbdbdbdb} & $\mathbb{R}P^4$ &$1 $ & $ (6;2,1,1) $ \\ 
    \texttt{gLALQQbbcdfeefffaaaaaaaaaawbwbwbwbwb} & $\mathbb{R}P^4$ &$1 $ & $ (4;1,1,1) $ \\
    \texttt{gLALQQbbcdeeffffaaaaaaaaaahahahahaha} & $\mathbb{R}P^4$ &$1 $ & $ (4;1,1,1) $ \\
    \bottomrule
  \end{tabular}
 \end{center}
\end{table}

\begin{proposition}
  The trisection genus of $S^3 \tilde{\times} S^1$ is two and the trisection genus of $\mathbb{R}P^4$ is four.
\end{proposition}

\begin{proof}
  This is an immediate corollary of the experiments summarised in Table~\ref{fig:altered}, and the statement of Lemma~\ref{lem:g}.
\end{proof}

\begin{lemma}
  \label{prop:twofour}
  Let $\tri$ be a triangulation of a $4$-manifold $M$ and let $\tri'$ be obtained from $\tri$ by a $2$-$4$-move. If $\tri$ is not tricolourable, then $\tri'$ is not tricolourable. If $\tri$ is tricolourable, then $\tri'$ is tricolourable if, up to permuting colours, the $2$-$4$-move either
  \begin{enumerate}
    \item removes a tetrahedron with colours $bbgg$, or
    \item removes a tetrahedron with colours $bgrr$ and inserts an edge with colours $bg$.
  \end{enumerate}
  For (1), either the genus of $H_r$ is increased by one or its number of connected components is decreased by one. For (2), $H_b$, $H_g$ and $H_r$ do not change. If $\tri''$ is obtained from $\tri$ by performing a $2$-$4$-move across every bi-coloured tetrahedron of $\tri$ (moves of type (1)), then $\tri''$ supports a trisection.
\end{lemma}

\begin{proof}
  First, assume that $\tri$ is not tricolourable (i.e., it does not admit a valid tricolouring on its set of vertices). Since a $2$-$4$-move does not change the vertex set of $\tri$ and only inserts new triangles, every invalid tricolouring on the vertex set of $\tri$ must be invalid on $\tri'$. 

  Hence, let $\tri$ admit a tricolouring. Then $\tri$ contains bi-coloured tetrahedra. Let $t$ be one of them with colours, say, $bbgg$. It follows, that $t$ is contained in exactly two pentachora $\Delta_1$ and $\Delta_2$ both coloured $bbggr$. Removing $t$ by a $2$-$4$-move necessarily introduces a mono-chromatic edge between the $r$-coloured vertices of $\Delta_1$ and $\Delta_2$.  In particular, the collection of handlebodies $H_r$ that is a regular neighbourhood of the $r$-coloured subgraph of $\tri$ either gains a handle or loses a connected component. 

  If the $2$-$4$-move is not performed at a bi-coloured tetrahedra, we can assume it removes a tetrahedron $t$ coloured $bbgr$ (up to permuting colours). It follows that $t$ is contained in two pentachora $\Delta_1$ and $\Delta_2$ coloured $bbggr$ or $bbgrr$. If the extra vertices of $\Delta_1$ and $\Delta_2$ have the same colour (that is, if a mono-chromatic edge is introduced by the $2$-$4$-move), then the $2$-$4$-move introduces a monochromatic triangle. If the extra vertices have distinct colours (necessarily $g$ and $r$) then the given tricolouring is still valid after the bistellar move. Moreover, the topology of the respective collections of $4$-dimensional handlebodies does not change under this operation.

  The last statement of the proposition follows from the proof of correctness of the algorithm from~\cite{Bell-computing-2017}. 
\end{proof}

Motivated by the last statement of Lemma~\ref{prop:twofour} (i.e.,~\cite{Bell-computing-2017}) we iteratively apply $2$-$4$-moves to the $13$ tricolourable $6$-pentachora triangulations of non-orientable $4$-manifolds. The results are summarised in Table~\ref{fig:altered}. Note that some of the trisections supported by triangulations after applying some $2$-$4$-moves are still of minimal genus. By Lemma~\ref{prop:twofour} this suggests that some bistellar moves to obtain this trisection must have been of type (2). This is in particular the case, if the bistellar move is performed on a c-tricolouring and thus a move of type (1) must increase the genus of one of the $4$-dimensional handlebodies. This motivates to look for small genus trisections by performing bistellar $2$-$4$-moves of type (2) on a (c-)tricolourable triangulation.

However, by running through all triangulations obtainable from one of the $13$ tricolourable census-triangulations by up to three $2$-$4$-moves we see that the range of genera of trisections supported by those triangulations is quite wide, see Table~\ref{fig:range}. Again, this is in-line with the statements of Lemma~\ref{prop:twofour}.

\begin{table}[bhpt]
 \begin{center}
  \caption{Range of genera of supported trisections of the $13$ non-orientable tricolourable census-triangulations after applying three bistellar $2$-$4$-moves in all possible ways. \label{fig:range}}
{\small
  \begin{tabular}{l|lll}
    \toprule
    iso. sig. & genera of trisections supported  \\
    \midrule
    \texttt{gLALQQbbbdefefffaaaaaaaaaa3b3b3b3b3b} &  $(4; 1, 1, 1),(6; 2, 1, 1),(8; 3, 1, 1),(8; 2, 2, 1)$\\
    \texttt{gLwMQQcceeeffeffaaaaaaaaaa9a9a9a9a9a} &  $(2; 1, 1, 1),(4; 2, 1, 1),(6; 3, 1, 1),(6; 2, 2, 1),(8; 3, 2, 1),(8; 2, 2, 2)$\\
    \texttt{gLMPMQccdeeeffffaaaa3aaaaaaaaaaaaa3a} &  $(2; 1, 1, 1),(4; 2, 1, 1),(6; 3, 1, 1),(6; 2, 2, 1)$\\
    \texttt{gLMPMQccdeeeffffaaaaabaaaaaaaaaaaaab} &  $(2; 1, 1, 1),(4; 2, 1, 1),(6; 3, 1, 1),(6; 2, 2, 1),(8; 3, 2, 1),(8; 2, 2, 2)$\\ 
    \texttt{gLLAQQcddcfdefffaaVbVbVbxbVbbaaaaaaa} &  $(4; 1, 1, 1),(6; 2, 1, 1),(8; 3, 1, 1),(8; 2, 2, 1),(10; 3, 2, 1),(10; 2, 2, 2)$\\
    \texttt{gLLAQQbdedfefeefaadbdbdbaaoaaadbdboa} &  $(6; 2, 1, 1),(8; 3, 1, 1),(8; 2, 2, 1),(10; 3, 2, 1),(10; 2, 2, 2)$\\
    \texttt{gLLAQQbdedffeeefaadbdbdbaaaavadbdbva} &  $(6; 2, 1, 1),(8; 3, 1, 1),(8; 2, 2, 1),(10; 3, 2, 1),(10; 2, 2, 2)$\\ 
    \texttt{gLAAMQacbdcdefffcaTava4acavaya1aYa2a} &  $(8; 2, 2, 2)$\\ 
    \texttt{gvLQQQcdefdfefefya2aqbPbgaGbjbSbvbba} &  $(4; 1, 1, 1),(6; 2, 1, 1),(8; 3, 1, 1),(8; 2, 2, 1),(10; 3, 2, 1),(10; 2, 2, 2)$\\
    \texttt{gLALQQbbbeffefefaaaaaaaaaasbsbsbsbsb} &  $(4; 1, 1, 1),(6; 2, 1, 1)$\\
    \texttt{gLAMPQbbcdeffeffaaaaaaaaaadbdbdbdbdb} &  $(6; 2, 1, 1),(8; 3, 1, 1),(8; 2, 2, 1),(10; 3, 2, 1),(10; 2, 2, 2)$\\ 
    \texttt{gLALQQbbcdfeefffaaaaaaaaaawbwbwbwbwb} &  $(4; 1, 1, 1),(6; 2, 1, 1),(8; 3, 1, 1),(8; 2, 2, 1)$\\
    \texttt{gLALQQbbcdeeffffaaaaaaaaaahahahahaha} &  $(4; 1, 1, 1),(6; 2, 1, 1),(8; 3, 1, 1),(8; 2, 2, 1)$\\
    \bottomrule
  \end{tabular}
}
 \end{center}
\end{table}

\section{Trisection genus for all known standard simply connected PL $4$-manifolds}
\label{sec:simply}

This section contains the proofs of Theorems~\ref{thm:main} and \ref{thm:2}. We start with a purely theoretical observation.

\begin{lemma}
  \label{lem:sc}
  Let $M$ be a simply connected $4$-manifold with second Betti number $\beta_2$, and let $\tri$ be a triangulation of $M$ coming from a simple crystallisation. Then $\tri$ admits $15$ c-tricolourings.

  Moreover, if some of these c-tricolourings are in fact ts-tricolourings, then their dual trisections must be of type $(\beta_2; 0,0,0)$. In particular they must all be minimal genus trisections.
\end{lemma}

\begin{proof}
  A triangulation of a $4$-manifold $M$ coming from a simple crystallisation is a simplicial cell complex $\tri$ such that a) $M \cong_{PL} |\tri|$, b) no face has self-identifications, and c) every pentachoron shares the same $5$ vertices and the same $10$ edges.

  Given such a triangulation $\tri$, property c) ensures, that there are $15$ ways to properly tricolour the $5$ vertices of $\tri$ (one colour colours a single vertex and the remaining two colours colour the remaining four vertices in pairs), and each one of them colours every pentachoron equally and thus in a valid way. Moreover, also because of property c), every colour either spans a single edge or a single vertex of $\tri$, and because of property b) a neighbourhood of this edge or vertex is diffeomeomorphic to a $4$-ball. In particular, $M$ decomposes into three $4$-balls (i.e., handlebodies of genus zero) and all possible $15$ tricolourings from above are in fact c-tricolourings.

  Hence, let us assume that the $3$-dimensional trisection submanifolds dual to some of the c-tricolourings are handlebodies -- and thus some of the c-tricolourings are, in fact, ts-tricolourings. It remains to determine the genus of the central surface of the trisection dual to these ts-tricolourings.

  The $f$-vector of $\tri$ is given by $$f(\tri) = (5,10,10\beta_2+10,15\beta_2+5,6\beta_2+2),\footnote{Property c) in the definition of a simple crystallisation states that $\tri$ must have $5$ vertices and $10$ edges. Since all ten edges of $\tri$ are contained in a single pentachoron, $M \cong_{PL} |\tri|$ must be simply connected and thus its Euler characteristic is given by $\chi(M) = 2+\beta_2 (M,\mathbb{Z})$. The other entries of the $f$-vector are then determined by the Dehn--Sommerville equations for $4$-manifolds.}$$ where $\beta_2 = \dim (H_2 (M,\mathbb{Z}))$. Every pair of the $5$ vertices spans exactly one of the $10$ edges, and each of the $10$ boundaries of triangles spanned by the $10$ edges of $\tri$ bounds exactly $(\beta_2+1)$ parallel copies of triangles. It follows that there are exactly $4\beta_2 +4$ tricoloured triangles. The central surface is thus spanned by $6\beta_2 +2$ quadrilaterals, $4\beta_2+4$ vertices, and is (orientable) of Euler characteristic $2-2\beta_2$. Hence, its genus is equal to the second Betti number of~$M$.

  The trisection must be a minimal genus trisection for (at least) two reasons: (i) all of its $4$-dimensional handlebodies are of genus zero (Lemma~\ref{lem:000}), and (ii) the genus of its central surface equals the second Betti number of $M$ (Lemma~\ref{lem:g}).
\end{proof}

\begin{proof}[Proof of Theorem~\ref{thm:main}]
  In \cite{BaSCryst}, Basak and the first author constructively show that there exists a simple crystallisation of the $K3$ surface. The result now follows from Lemma~\ref{lem:sc} together with a check of the collapsibility of the $3$-dimensional trisection submanifolds for a particular example triangulation coming from a simple crystallisation.

  The isomorphism signature given in Appendix~\ref{app:isosigs} belongs to such a singular triangulation of the $K3$ surface supporting a trisection.
\end{proof}

\begin{proof}[Proof of Theorem~\ref{thm:2}]
  Given an arbitrary standard simply connected $4$-manifold $M$ and a simple crystallisation of $M$, Lemma~\ref{lem:sc} tells us that its associated triangulation $\tri$ admits $15$ c-tricolourings.
  It remains to show for a particular such triangulation of $M$ and its $15$ c-tricolourings that every $3$-dimensional trisection submanifold is a handlebody, that is, it retracts to a $1$-dimensional complex.

  For this, w.l.o.g. assume that for every choice of c-tricolouring, the colour colouring only one vertex is blue $b$, and the colours colouring two vertices are red $r$ and green $g$. By construction, the two $3$-dimensional trisection submanifolds defined by $b$ and $r$ ($b$ and $g$) retract to the (multi-)graph whose vertices are the mid-points of the two edges with endpoints coloured by $b$ and $r$ ($b$ and $g$), and whose edges are normal arcs parallel to the monochromatic edge in the $\beta_2+1$ triangles coloured $rrb$ ($ggb$ respectively). In particular, these two $3$-dimensional trisection submanifolds retract to a $1$-dimensional complex of genus~$\beta_2$.

  The third $3$-dimensional trisection submanifold defined by $r$ and $g$ initially retracts to a $2$-dimensional subcomplex $Q_{rg}$ with $4$ vertices, one for every $rg$-coloured edge, $4\beta_2 + 4$ edges, one for each $rrg$- and each $rgg$-coloured triangle, and $3\beta_2+1$ quadrilaterals, one for each $rrgg$-coloured tetrahedron. This $2$-dimensional complex might or might not continue to collapse to a $1$-dimensional complex depending on the combinatorial properties of $\tri$. If $Q_{rg}$ collapses, however, it must collapse to a $1$-dimensional complex of genus~$\beta_2$.

  The unique simple crystallisation of $\mathbb{C}P^2$ as well as $266$ of the $267$ simple crystallisations of $S^2 \times S^2$ translate to triangulations where all of the $15$ c-tricolourings (guaranteed by Lemma~\ref{lem:sc}) are in fact ts-tricolourings. Moreover, there exist various simple triangulations of the $K3$ surface with this property. See \cite{BaSCryst} for pictures of representative simple crystallisations of $\mathbb{C}P^2$, $S^2 \times S^2$ and $K3$. These particular simple crystallisations turn out to have associated triangulations with $15$ ts-tricolourings. See Appendix~\ref{app:isosigs} for their isomorphism signatures.

  Assume that for all connected sums $N$ of these three prime simply connected $4$-manifolds with arbitrary orientation and second Betti number $\leq k$, there always exists a triangulation coming from a simple crystallisation with all $15$ c-tricolourings being also ts-tricolourings.

  It remains to show that for two such manifolds $N_1$ and $N_2$, there exists a triangulation coming from a simple crystallisation of $N_1 \# N_2$ with this property. By the induction hypothesis we can assume that there exist such triangulations $\tri_i$ of $N_i$, $ i = 1,2$, with all $15$ c-tricolourings producing $2$-complexes $Q_{rg} (\tri_i)$, $i = 1,2$, collapsing to a $1$-dimensional complex.

  Fix a ts-tricolouring on $\tri_1$ and a collapsing sequence of the quadrilaterals of $Q_{rg} (\tri_1)$. Remove one of the two pentachora containing the quadrilateral which is collapsed last. Similarly, fix a ts-tricolouring and collapsing sequence on $\tri_2$ and remove one of the two pentachora containing the quadrilateral removed first. Glue together both triangulations along their boundaries such that the edges through which the last quadrilateral of $Q_{rg} (\tri_1)$ and the first quadrilateral of $Q_{rg} (\tri_2)$ are collapsed, are aligned. For this, colours $r$ and $g$ of $\tri_2$ might need to be swapped (note that such a swap in colours does not change the ts-tricolouring class as such). The collapsing sequence can now be concatenated, yielding a collapsing sequence for~$Q_{rg} (\tri_1 \# \tri_2)$.

  Moreover, the property of a triangulation of coming from a simple crystallisation is preserved under taking connected sums of the type as described above \cite{BaSCryst}. Hence, repeating this process for all $15$ tricolourings finishes the proof.
\end{proof}

\bibliographystyle{plain}
\bibliography{trisection_genus}

\begin{thebibliography}{10}

\bibitem{BaSCryst}
Biplab Basak and Jonathan Spreer.
\newblock Simple crystallizations of $4$-manifolds.
\newblock {\em Adv. in Geom.}, 16(1):111--130, 2016.

\bibitem{Bell-computing-2017}
M.~{Bell}, J.~{Hass}, J.~{Hyam Rubinstein}, and S.~{Tillmann}.
\newblock Computing trisections of 4-manifolds.
\newblock \texttt{arXiv:1711.02763 [math.GT]}, 2017.

\bibitem{Benedetti13RandomDMT}
Bruno Benedetti and Frank~H. Lutz.
\newblock Random discrete {M}orse theory and a new library of triangulations.
\newblock {\em Exp. Math.}, 23(1):66--94, 2014.

\bibitem{4d-census}
Benjamin~A. Burton and Ryan Budney.
\newblock A census of small triangulated 4-manifolds.
\newblock in preparation, $\ge$2018.

\bibitem{regina}
Benjamin~A. Burton, Ryan Budney, William Pettersson, et~al.
\newblock Regina: Software for low-dimensional topology.
\newblock {\tt http://\allowbreak regina-normal.\allowbreak github.\allowbreak
  io/}, 1999--2017.

\bibitem{C1}
Stewart~S. Cairns.
\newblock Introduction of a {R}iemannian geometry on a triangulable 4-manifold.
\newblock {\em Ann. of Math. (2)}, 45:218--219, 1944.

\bibitem{C2}
Stewart~S. Cairns.
\newblock The manifold smoothing problem.
\newblock {\em Bull. Amer. Math. Soc.}, 67:237--238, 1961.

\bibitem{Chu18TrisectionGenus}
Michelle Chu and Stephan Tillmann.
\newblock {Reflections on trisection genus}.
\newblock \texttt{arXiv:1809.04801 [math.GT]}, 2018.

\bibitem{Donaldson83GaugeTheory4Mflds}
S.~K. Donaldson.
\newblock {A}n application of gauge theory to four-dimensional topology.
\newblock {\em J. Differential Geom.}, 18(2):279--315, 1983.

\bibitem{Feller-calculating-2017}
P.~{Feller}, M.~{Klug}, T.~{Schirmer}, and D.~{Zemke}.
\newblock Calculating the homology and intersection form of a 4-manifold from a
  trisection diagram.
\newblock \texttt{arXiv:1711.04762 [math.GT]}, 2017.

\bibitem{Ferri86GraphTheoryCrystallizations}
M.~Ferri, C.~Gagliardi, and L.~Grasselli.
\newblock A graph-theoretical representation of {PL}-manifolds---a survey on
  crystallizations.
\newblock {\em Aequationes Math.}, 31(2-3):121--141, 1986.

\bibitem{Freedman82Top4DimMnf}
M.~Freedman.
\newblock {T}he topology of four-dimensional manifolds.
\newblock {\em J. Differential Geom.}, 17:357--453, 1982.

\bibitem{Furuta01MonopoleEq}
M.~Furuta.
\newblock Monopole equation and the {$\frac{11}8$}-conjecture.
\newblock {\em Math. Res. Lett.}, 8(3):279--291, 2001.

\bibitem{GK}
David Gay and Robion Kirby.
\newblock Trisecting 4-manifolds.
\newblock {\em Geom. Topol.}, 20(6):3097--3132, 2016.

\bibitem{Gay2015}
David~T. Gay.
\newblock Trisections of {L}efschetz pencils.
\newblock {\em Algebr. Geom. Topol.}, 16(6):3523--3531, 2016.

\bibitem{Hatcher2002AlgTop}
A.~Hatcher.
\newblock {\em Algebraic Topology}.
\newblock Cambridge University Press, 2002.

\bibitem{MSZ}
Jeffrey Meier, Trent Schirmer, and Alexander Zupan.
\newblock Classification of trisections and the generalized property {R}
  conjecture.
\newblock {\em Proc. Amer. Math. Soc.}, 144(11):4983--4997, 2016.

\bibitem{Meier-classification-2016}
Jeffrey Meier, Trent Schirmer, and Alexander Zupan.
\newblock Classification of trisections and the generalized property {R}
  conjecture.
\newblock {\em Proc. Amer. Math. Soc.}, 144(11):4983--4997, 2016.

\bibitem{MZ-bridge}
Jeffrey Meier and Alexander Zupan.
\newblock Bridge trisections of knotted surfaces in {$S^4$}.
\newblock {\em Trans. Amer. Math. Soc.}, 369(10):7343--7386, 2017.

\bibitem{MZ-standard}
Jeffrey Meier and Alexander Zupan.
\newblock Genus-two trisections are standard.
\newblock {\em Geom. Topol.}, 21(3):1583--1630, 2017.

\bibitem{Milnor73SymmBilForms}
John Milnor and Dale Husemoller.
\newblock {\em {S}ymmetric bilinear forms}, volume~73 of {\em {E}rgebnisse der
  {M}athematik und ihrer {G}renzgebiete}.
\newblock Springer-Verlag, New York, 1973.

\bibitem{Pezzana74Crystallizations}
Mario Pezzana.
\newblock Sulla struttura topologica delle variet\`a compatte.
\newblock {\em Ati Sem. Mat. Fis. Univ. Modena}, 23(1):269--277 (1975), 1974.

\bibitem{Rohlin84NewResults4Mflds}
V.~A. Rohlin.
\newblock {N}ew results in the theory of four-dimensional manifolds.
\newblock {\em Doklady Akad. Nauk SSSR (N.S.)}, 84:221--224, 1952.

\bibitem{Rubinstein-multisections-2016}
J.~Hyam {Rubinstein} and S.~{Tillmann}.
\newblock Multisections of piecewise linear manifolds.
\newblock \texttt{arXiv:1602.03279 [math.GT]}, 2016.

\bibitem{Saveliev12Lectures}
Nikolai Saveliev.
\newblock {\em Lectures on the topology of 3-manifolds}.
\newblock De Gruyter Textbook. Walter de Gruyter \& Co., Berlin, revised
  edition, 2012.
\newblock An introduction to the Casson invariant.

\bibitem{socg}
Jonathan Spreer and Stephan Tillmann.
\newblock {The trisection genus of the standard simply connected PL
  4-manifolds}.
\newblock In Bettina Speckmann and Csaba~D. T{\'o}th, editors, {\em 34th
  International Symposium on Computational Geometry (SoCG 2018)}, volume~99 of
  {\em Leibniz International Proceedings in Informatics (LIPIcs)}, pages
  71:1--71:13, Dagstuhl, Germany, 2018. Schloss Dagstuhl--Leibniz-Zentrum fuer
  Informatik.

\bibitem{Tancer12CollNPComplete}
Martin Tancer.
\newblock Recognition of collapsible complexes is np-complete.
\newblock {\em Discrete {\&} Computational Geometry}, 55(1):21--38, Jan 2016.

\bibitem{Whitehead1940}
J.~H.~C. Whitehead.
\newblock On {$C^1$}-complexes.
\newblock {\em Ann. of Math. (2)}, 41:809--824, 1940.

\end{thebibliography}
\newpage
\appendix

\section{Isomorphism signatures of triangulations of $S^4$, $\mathbb{C}P^2$, $S^2 \times S^2$ and the $K3$ surface}
\label{app:isosigs}

Isomorphism signatures of the triangulations associated to three simple crystallisations of $\mathbb{C}P^2$, $S^2 \times S^2$ and the $K3$ surface, and a triangulation homeomorphic to $S^4$. Each of them admits $15$ ts-tricolourings (see Section~\ref{sec:simply} for details, see \cite{BaSCryst} for pictures).

To construct the triangulations download \emph{Regina} \cite{regina}, and produce a new $4$-manifold triangulation by selecting type of triangulation ``From isomorphism signature'' and pasting in one of the strings given below (please note that you need to remove newline characters before pasting in the isomorphism signature of the $K3$ surface).

\subsection*{Unique $8$-pentachora simple crystallisation of $\mathbb{C}P^2$}

\begin{verbatim}
iLvLQQQkbghhghhfffggfaaaaaaaaaaaaaaaaaaaaaaaaaa
\end{verbatim}

\subsection*{$14$-pentachora simple crystallisation of $S^2 \times S^2$}

\begin{verbatim}
oLvMPLQAPMQPkbfgghjihhiilkkmllmnnnnaaaaaaaaaaaaaaaaaaaaaaaaaaaaaaaaaaaaaaaaaaaa
\end{verbatim}

\subsection*{$134$-pentachora simple crystallisation of the $K3$ surface}

\begin{verbatim}
-cgcLvLALLvvwwzvzMvwLAwvwvvwvwAAPvPPwzQQwMvPAQLQzQzPwLvwvPzvPwQLAPAPQMQwMQQQQAQ
zQQQQQAQPQQQQQQQQQwzvQQMMMMQQQQQQQQPkcahafafakaiasauaxapaBaDaGaKaDaFaFaQaWa3a6a
1aTa7aYaSaTa-aVa4a9aVaab9abbabPaPaZaibPagbfb5a3a5a1aZa0a2aOa-aVa-a8ahbQaubvbubz
bAbGbIbHbHbNbObQbQbKbLbRbHbvbLbwbMbTbTbububWbGbXbHbObRbDbYbNbYbKbwbObSbGbIbVbMb
MbZbZbGbVb1bWbxbTbPbWbXb0b3bPbPb3bCbCb3bYb0bJbybybJb2bIbvbIbvb3bzbxbzb5bWb1b5bz
bAbDbRb4bFbDbSb2bUbUb2bUbUb4bKbCbybybCb8bac-b+b+b+b9b9b8b8bbc6bdc6becacbc9b9bdc
6bec-b7bdc-b-bdcecccaccc7b+b8bbc6b7b7bfcfcfcfcaaaaaaaaaaaaaaaaaaaaaaaaaaaaaaaaa
aaaaaaaaaaaaaaaaaaaaaaaaaaaaaaaaaaaaaaaaaaaaaaaaaaaaaaaaaaaaaaaaaaaaaaaaaaaaaaa
aaaaaaaaaaaaaaaaaaaaaaaaaaaaaaaaaaaaaaaaaaaaaaaaaaaaaaaaaaaaaaaaaaaaaaaaaaaaaaa
aaaaaaaaaaaaaaaaaaaaaaaaaaaaaaaaaaaaaaaaaaaaaaaaaaaaaaaaaaaaaaaaaaaaaaaaaaaaaaa
aaaaaaaaaaaaaaaaaaaaaaaaaaaaaaaaaaaaaaaaaaaaaaaaaaaaaaaaaaaaaaaaaaaaaaaaaaaaaaa
aaaaaaaaaaaaaaaaaaaaaaaaaaaaaaaaaaaaaaaaaaaaaaaaaaaaaaa
\end{verbatim}

\subsection*{Triangulation homeomorphic to $S^4$ supporting trisections of multiple types}

Triangulation homeomorphic to $S^4$ with $15$ ts-tricolourings, supporting trisections of type $(0;0,0,0)$ ($\times 10$), $(1;1,0,0)$ ($\times 4$), and $(2;1,1,0)$  ($\times 1$). See Section~\ref{ssec:expts} for details.

\begin{verbatim}
gLAAMQbbcddeffffaaaaaaaaaaaaaaaaaaaa
\end{verbatim}

\end{document}